\documentclass[11pt]{amsart}

\usepackage[pdftex]{graphicx}
\usepackage{amssymb}
\usepackage[colorlinks,urlcolor=blue]{hyperref}
\usepackage{amsmath}
\usepackage{amsfonts}
\usepackage{amsthm}
\usepackage{amssymb}
\usepackage{mathrsfs}

\numberwithin{equation}{section}

\newtheorem{theorem}{Theorem}[section]

\newtheorem{proposition}[theorem]{Proposition}
\newtheorem{corollary}[theorem]{Corollary}

\theoremstyle{definition}
\newtheorem{definition}[theorem]{Definition}
\theoremstyle{remark}
\newtheorem{remark}[theorem]{Remark}

\newcommand{\ord}{\mathrm{ord}}
\newcommand{\mult}{\mathrm{mult}}

\DeclareMathOperator{\area}{\mathrm{Area}}
\DeclareMathOperator{\ind}{\mathrm{ind}}
\DeclareMathOperator{\nul}{\mathrm{nul}}
\DeclareMathOperator{\Hom}{\mathrm{Hom}}

\begin{document}

\title[Yang-Yau inequality]{On the Yang-Yau inequality for the first Laplace eigenvalue}
\begin{abstract}
In a seminal paper published in 1980, P.~C.~Yang and S.-T.~Yau proved an inequality bounding the first eigenvalue of the Laplacian on an orientable Riemannian surface in terms of its genus $\gamma$ and the area. The equality in Yang-Yau's estimate is attained for $\gamma=0$ by an old result of J.~Hersch and it was recently shown by S.~Nayatani and T.~Shoda that it is also attained for $\gamma=2$. In the present article we combine techniques from algebraic geometry and minimal surface theory to show that Yang-Yau's inequality is strict for all genera $\gamma >2$. Previously this was only known for $\gamma=1$. In the second part of the paper we apply Chern-Wolfson's notion of harmonic sequence to obtain an upper bound on the total branching order of harmonic maps from surfaces to spheres. Applications of these results to extremal metrics for eigenvalues are discussed. 
\end{abstract}
\author[M. Karpukhin]{Mikhail Karpukhin}
\address{Department of Mathematics,
University of California, Irvine, 340 Rowland Hall, Irvine, CA 92697-3875
}
\email{mkarpukh@uci.edu}
\date{}
\maketitle

\section{Introduction}

\subsection{Yang-Yau inequality}
Let $(M,g)$ be a closed Riemannian surface. The associated Laplace-Beltrami operator $\Delta_g$ on the space of smooth functions is defined as $\Delta_g = \delta\circ d$, where $d\colon C^\infty(M)\to\Omega^1(M)$ is the differential and $\delta$ is its formal adjoint.  
As an unbounded operator on $L^2(M)$, the operator $\Delta_g$ has discrete non-negative spectrum
$$
0 = \lambda_0(M,g) < \lambda_1(M,g) \leqslant \lambda_2(M,g) \leqslant \lambda_3(M,g) \leqslant \ldots,
$$
where eigenvalues are written with multiplicities.

Let us consider the normalized eigenvalues 
$$
\bar\lambda_k(M,g) = \lambda_k(M,g)\area_g(M).
$$
Then for any positive number $t>0$ one has $\bar\lambda_k(M,tg) = \bar\lambda_k(M,g)$. Furthermore, one defines
$$
\Lambda_k(M,g) = \sup_g\bar\lambda_k(M,g),
$$
where the supremum is taken over all possible Riemannian metrics on $M$. One of the fundamental problems in spectral geometry is to determine the value of $\Lambda_k(M)$ and to find the metrics $g$ for which this value is attained. Let us review the current progress on this problem. For a more detailed survey, see the introduction to the paper~\cite{KNPP}.

In the paper~\cite{YY} Yang and Yau proved that for any orientable surface $M$ of genus $\gamma$ one has
$$
\Lambda_1(M)\leqslant 8\pi(\gamma + 1).
$$
However, it was remarked in~\cite{ESI_YY} that the same proof yields the following stronger inequality,
\begin{equation}
\label{YYineq}
\Lambda_1(M)\leqslant 8\pi\left[\frac{\gamma+3}{2}\right],
\end{equation}
where $ [x]$ denotes the integer part of $x$.
In the following we refer to~\eqref{YYineq} as Yang-Yau inequality. We note that an alternative proof of~\eqref{YYineq} using the concept of conformal volume was given by Li and Yau in~\cite{LiYau}.

Prior to the paper~\cite{YY}, it was known to Hersch~\cite{Hersch} that the equality in inequality~\eqref{YYineq} is attained for $\gamma=0$.
Later, Nadirashvili proved in~\cite{Nadirashvili_torus} that $\Lambda_1(\mathbb{T}^2) = \frac{8\pi^2}{\sqrt{3}}$. His proof implicitly relies on the fact that the inequality~\eqref{YYineq} is strict for $\gamma=1$, see~\cite{CKM}. It was conjectured in~\cite{JLNNP} that the equality in~\eqref{YYineq} is attained for $\gamma=2$. Very recently it was proved by Nayatani and Shoda in the paper~\cite{NayaSho}. In the first result of the present paper we prove that $\gamma=0$  and $\gamma=2$ are the only possible values of $\gamma$ for which the equality in inequality~\eqref{YYineq} can be attained.

\begin{theorem}
\label{YYth}
For any orientable surface $M$ of genus $\gamma>2$ one has
$$
\Lambda_1(M)< 8\pi\left[\frac{\gamma+3}{2}\right].
$$
\end{theorem}

The proof of Theorem~\ref{YYth} is inspired by an argument of Ros~\cite[Theorem 14]{Ros}. His results imply Theorem~\ref{YYth} in case $\gamma=4$. Our proof is an extension of his ideas and relies on a combination of algebraic techniques (geometry of special divisors) and minimal surface theory (index bounds for branched multivalued minimal immersions). 

\begin{remark}
In the paper~\cite{Karpukhin_nor} the author has shown an analog of inequality~\eqref{YYineq} for non-orientable surfaces. Namely, if $M$ is a non-orientable surface of genus $\gamma$, then 
\begin{equation*}
\Lambda_1(M)\leqslant 16\pi\left[\frac{\gamma+3}{2}\right].
\end{equation*}
The genus of a non-orientable surface is defined to be the genus of its orientable double cover. 
We remark that a slight modification of the argument in~\cite{Karpukhin_nor} yields the strict inequality
\begin{equation*}
\Lambda_1(M)< 16\pi\left[\frac{\gamma+3}{2}\right].
\end{equation*}

Indeed, the argument relies on the fact that for any metric $g$ on the disk $\mathbb{D}$ one has $\bar\lambda_1^N(\mathbb{D},g)\leqslant 8\pi$, where $\bar\lambda_1^N$ stands for the first normalized Neumann eigenvalue. However, according to~\cite[Theorem 1]{LiYau} the equality is possible only if $(\mathbb{D},g)$ admits an isometric minimal immersion by the first Neumann eigenfunctions. Let $F = (f_1,\ldots,f_n)$ be such an immersion. Since $F$ is an isometry, this implies that the length $|\nu|$ of the normal vector to $\partial \mathbb{D}$ satisfies $|\nu|^2 = \left(\sum df_i\otimes df_i\right)(\nu,\nu) = 0$, which is a contradiction. Therefore, the equality is impossible, i.e. $\bar\lambda_1^N(\mathbb{D},g)< 8\pi$.
\end{remark}
\begin{remark}
\label{Korevaar_remark}
In~\cite{Korevaar} (see also~\cite{Hassannezhad}) a version of the Yang-Yau inequality is proved for higher eigenvalues. Namely, it is shown that there exists a universal constant $C$ such that for all surfaces of genus $\gamma$ one has
$$
\Lambda_k(M,g)\leqslant Ck(\gamma+1).
$$
\end{remark}

\subsection{Connection to harmonic maps} One of the motivations for studying the extremal properties of functionals $\bar\lambda_k(M,g)$ is the connection to harmonic and minimal maps which we briefly recall below. First, let us set
$$
\Lambda_k(M,[g]) = \sup_{\tilde g\in [g]}\bar\lambda_k(M,\tilde g),
$$
where $[g]$ is a class of metrics conformal to $g$. For the purposes of our discussion it is convenient to allow $g$ to have conical singularities at isolated points of $M$. Therefore, we set $[g] = \{\tilde g|\,\tilde g = f^2g\}$, where $f$ ranges over non-negative smooth functions which are allowed to be equal to zero at isolated points.

The functional $\bar\lambda_i(M,g)$ depends continuously  on the metric $g$, but this functional is not differentiable. 
However, it is shown in the papers~\cite{BU, Berger, ElSoufiIlias1} that for an analytic family of metrics $g_t$ there exist the left and the right 
derivatives
with respect to $t$. This is a motivation for the following definition, see the papers~\cite{ElSoufiIlias1,Nadirashvili_torus}.

\begin{definition} A Riemannian metric $g$ on a closed surface $M$ is called {\em extremal for the functional} 
$\bar\lambda_i$ if for any analytic deformation $g_t$ such that $g_0 = g$ the following inequality holds,
$$
\frac{d}{dt}\bar\lambda_i(M,g_t)\Bigl|_{t=0+} \times \frac{d}{dt}\bar\lambda_i(M,g_t)\Bigl|_{t=0-}\leqslant 0.
$$

Similarly, $g$ is called {\em conformally extremal} if the same inequality holds for conformal deformations, i.e. for deformations satisfying $[g_t]=[g]$ for all $t$.
\end{definition}

Let $\Phi\colon (M,g)\to\mathbb{S}^n$ be a harmonic map and let $h = \frac{1}{2}|\nabla\Phi|_g^2\,g$ be a metric with isolated conical singularities in the conformal class of $g$,~i.e. $h\in [g]$. Here and everywhere below, the sphere $\mathbb{S}^n$ is considered to be the unit sphere in $\mathbb{R}^{n+1}$ endowed with the induced metric $g_{\mathbb{S}^n}$. The metric $h$ has conical singularities at zeroes of $d\Phi$, which are isolated by harmonicity of $\Phi$. If $\Phi$ is conformal then $h=\Phi^*g_{\mathbb{S}^n}$.
Let us introduce the Weyl's counting function
$$
N_h(\lambda) = \#\{i|\,\lambda_i(M,h)<\lambda\}.
$$

\begin{theorem}[Nadirashvili~\cite{Nadirashvili_torus}, El Soufi and Ilias,~\cite{ElSoufiIlias1}, see also~\cite{KNPP}] 
\label{extremal_th}
If $g$ is conformally extremal for the functional $\bar\lambda_i(M,g)$ then there exists a harmonic map $\Psi\colon (M,g)\to\mathbb{S}^n$ whose components are $\lambda_i$-eigenfunctions. If $g$ is extremal, then $\Psi$ can be chosen to be conformal.

Conversely, if $\Phi\colon(M,g)\to\mathbb{S}^n$ is a harmonic map to the unit sphere, then the metric $h = \frac{1}{2}|\nabla\Phi|_g^2\,g$ is conformally extremal for the functional $\bar\lambda_{N(2)}$. Furthermore, if $\Phi$ is conformal, i.e. if $\Phi\colon (M,h)\to\mathbb{S}^n$ is a branched minimal immersion, then $h$ is extremal for the functional $\bar\lambda_{N_h(2)}$.
\end{theorem}

In particular, if there exists a metric realizing the quantities $\Lambda_k(M)$ or $\Lambda_k(M,[g])$, then this metric is extremal or conformally extremal respectively. Such metrics are called {\em maximal for} $\bar\lambda_k$ or {\em conformally maximal for} $\bar\lambda_k$ respectively. The existence of such metrics was studied in~\cite{Nadirashvili_torus, Petrides1, Petrides2}. For the sake of brevity we only state the following result.
\begin{theorem}[Petrides~\cite{Petrides1}]
\label{existence_th}
For any conformal class $[g]$ on a surface $M$ there exists a smooth metric $h$, possibly with isolated conical singularities, such that $\bar\lambda_1(M,h) = \Lambda_1(M,[g])$. As a result, there exists a harmonic map $\Phi\colon (M,g)\to\mathbb{S}^n$ such that $h=\frac{1}{2}|\nabla\Phi|^2_gg$. 
\end{theorem}
\begin{remark}
The existence result for higher eigenvalues requires an additional a priori condition that is not always satisfied, see~\cite{Petrides2} for details. For results concerning the values $\Lambda_k(M)$ for $k\geqslant 2$ we refer to~\cite{KRP2, KNPP, NP}.
\end{remark}

\subsection{Total branching order}
As we see from Theorems~\ref{extremal_th} and~\ref{existence_th}, metrics with conical singularities appear naturally in the study of extremal values of the functionals $\bar\lambda_k(M,g)$. However, one would expect that solutions of extremal problems have nice smoothness properties, in particular, one expects that extremal metrics should have few singular points.

The second result of the present article is an upper bound on the number of conical singularities. The most general version of the bound can be found in Section~\ref{HSS2}. The following theorem is less general but is easier to state.
\begin{theorem}
\label{Branching_th}
Let $\Phi\colon(M,g)\to\mathbb{S}^n$ be a linearly full harmonic map, i.e. its image is not contained in the sphere of smaller dimension. Let $b$ be a total number of branching points of $\Phi$ counted with multiplicity. 

If $\Phi$ is not totally isotropic then
\begin{equation}
\label{branching_ineq}
b\leqslant -\frac{1}{2}(n+1)\chi(M).
\end{equation}

If $\Phi$ is totally isotropic then $n=2m$ and
\begin{equation}
\label{branching_ineq2}
b\leqslant \frac{1}{2\pi m}E_g(\Phi) -\frac{1}{2}(m+1)\chi(M),
\end{equation}
where 
$$
E_g(\Phi) = \frac{1}{2}\int\limits_M |d\Phi^2|_g\,dV_g
$$
is the energy of the map $\Phi$.
\end{theorem}
For the definition of a totally isotropic map see Section~\ref{HSS2}. For now, we recall that totally isotropic maps are conformal, i.e. $E_g(\Phi)=\area_{\Phi^*g_{\mathbb{S}^n}}(M)$.
\begin{remark}
We note that the inequality~\eqref{branching_ineq2} becomes an equality if $m=1$. In that case,~\eqref{branching_ineq2} is a classical Riemann-Hurwitz formula for holomorphic maps $\Phi\colon M\to \mathbb{S}^2$.
\end{remark}
\begin{remark}
We remark that the bounds similar to~\eqref{branching_ineq} were obtained in~\cite{ET} for $n=4$. The paper~\cite{ET} also contains bounds on total branching order for other types of maps, including CMC-surfaces in space forms. Our approach is inspired by the argument in~\cite[Theorem 7.1]{Kusner}.
\end{remark}

\begin{corollary}
If $M\cong\mathbb{T}^2$ or $M\cong\mathbb{KL}$, where $\mathbb{KL}$ stands for the Klein bottle,  and $\Phi$ is not totally isotropic, then $\Phi$ does not have branch points.
\end{corollary}

We record several corollaries of the results in Section~\ref{HSS2} which could be of interest for applications. The following proposition is proved in Section~\ref{CaseA:sec}.
\begin{proposition}
\label{nonconformal:prop}
Suppose that $\Phi\colon (M,g)\to\mathbb{S}^n$ is a non-conformal harmonic map. Then
$$
b\leqslant -\chi(M).
$$
\end{proposition}

By Theorem~\ref{extremal_th} conformally extremal metrics for $\bar\lambda_k(M,g)$ have the form $h = \frac{1}{2}|\nabla\Phi|_g^2g$ for a harmonic map $\Phi\colon (M,g)\to\mathbb{S}^n$. The total number of branching points $b$ of $\Phi$ can be interpreted in terms of $h$ as follows. The point $p$ is a branching point of $\Phi$ of order $k$ (i.e. $d\Phi$ has a zero of order $k$ at $p$) iff $h$ has a conical singularity at $p$ of conical angle $2\pi(k+1)$. Thus, $b$ can be interpreted as a total number of conical singularities of $h$ counted with multiplicity. The following proposition is proved in Section~\ref{extremalbranching_sec}.

\begin{proposition}
\label{extremalbranching_prop}
There exists a universal constant $C$ such that for any metric $h$ conformally extremal for the functional $\bar\lambda_k(M,\cdot)$ one has that the total number $b$ of conical singularities of $h$ counted with multiplicities satisfies
$$b\leqslant C(\gamma+1)(\gamma+k),$$
where $\gamma$ is the genus of $M$.
\end{proposition}

\begin{remark}
We note that for $k=1$ the proof yields an explicit expression for $C$. In general, $C$ will depend on the constant in Remark~\ref{Korevaar_remark}.
\end{remark}
\begin{remark}
For a similar statement concerning extremal metrics for the first eigenvalue of the Dirac operator see~\cite[Theorem 1.1(C)]{Ammann}
\end{remark}
%

\subsection{Discussion}
It was noticed in~\cite[Theorem 1.4]{CKM} that properties of $\bar\lambda_1$-maximal metrics induced by branched minimal immersions to $\mathbb{S}^2$ differ significantly from those induced from $\mathbb{S}^n$ with $n\geqslant 3$. At the same time, as we will see in Section~\ref{Proof_section}, if the equality in Yang-Yau inequality is attained for a particular value of $\gamma$, then $\bar\lambda_1$-maximal metrics on an orientable surface of genus $\gamma$ are induced from the two-dimensional sphere. This is our primary motivation in studying the equality case of inequality~\eqref{YYineq}. In particular, Theorem~\ref{YYth} implies the following proposition.
\begin{proposition}
Let $\Sigma_\gamma$ denote an orientable surface of genus $\gamma>2$. Assume that 
\begin{equation}
\label{notS2}
\Lambda_1(\Sigma_\gamma)>8\pi\left[\frac{\gamma+1}{2}\right].
\end{equation}
Then any $\bar\lambda_1$-maximal metric on $\Sigma_\gamma$ (if such a metric exist) is induced by a branched minimal immersion to $\mathbb{S}^n$ with $n>2$.
\end{proposition}
\begin{proof}
Assume the contrary, i.e. that~\eqref{notS2} holds and there is a $\bar\lambda_1$-maximal metric on $\Sigma_\gamma$ which is induced by a branched minimal immersion to $\mathbb{S}^2$. The latter implies that $\Lambda_1(\Sigma_\gamma)$ is an integer multiple of $8\pi$. At the same time, Theorem~\ref{YYth} and inequality~\eqref{notS2} imply
$$
8\pi\left[\frac{\gamma+1}{2}\right]<\Lambda_1(\Sigma_\gamma)<8\pi\left[\frac{\gamma+3}{2}\right],
$$
i.e. $\Lambda_1(\Sigma_\gamma)$ is squeezed between two consequent multiples of $8\pi$. We arrive at a contradiction. 
\end{proof}

Note that by the results of Petrides~\cite{Petrides1}, $\bar\lambda_1(\Sigma_\gamma,g)$-maximal metrics smooth outside conical singularities exist provided $\Lambda_1(\Sigma_\gamma)>\Lambda_1(\Sigma_{\gamma-1})$. If $\gamma=3$, it was shown in~\cite{Matthiesen_thesis} that $\Lambda_1(\Sigma_3)>\Lambda_1(\Sigma_2)=16\pi$, which in this case guarantees both the existence of maximal metrics and condition~\eqref{notS2}. Therefore, $\bar\lambda_1$-maximal metrics on $\Sigma_3$ exist and are induced by a branched minimal immersion to $\mathbb{S}^n$ with $n>2$. At the same time, by~\cite{CES} one has $\Lambda_1(\Sigma_4)\geqslant \Lambda_1(\Sigma_3)>\Lambda_1(\Sigma_2)=16\pi$, thus, condition~\eqref{notS2} is satisfied for $\gamma=4$ as well. As a result, $\bar\lambda_1$-maximal metrics on $\Sigma_4$ (if exist) are induced by a branched minimal immersion to $\mathbb{S}^n$ with $n>2$.

\subsection{Outline of the proof} 
The proof of Theorem~\ref{YYth} has a lot of moving pieces. For convenience of the reader, we outline them below.

Assume that for a metric $g$ one has the equality in~\eqref{YYineq}. The starting point of the proof is an observation that the complex structure associated to $g$ admits a unique holomorphic map to $\mathbb{S}^2$ of degree $\left[\frac{\gamma+3}{2}\right]$, see Theorem~\ref{uniquepencil}. At the same time, Brill-Noether theory implies that for $\gamma>2$ such complex structures are special in the moduli space of Riemann surfaces, see discussion after Theorem~\ref{ind1}. Dealing with these special complex structures is the main difficulty of the proof. It relies on an interplay between Laplace eigenvalues and index bounds for minimal submanifolds of $\mathbb{R}^3$ explained in Section~\ref{prelim_section}. It turns out that the particular properties of the complex structure associated to $g$ allows one to construct a branched multivalued minimal immersion to $\mathbb{R}^3$ and, as a result, apply the index bound of Section~\ref{index_section} to conclude the proof.

\subsection*{Plan of the paper} The paper is organized in the following way. In Section~\ref{prelim_section} we recall the relevant background, including algebraic geometry notations, a Schr\"odinger operator associated to a meromorphic function and metrics with conical singularities. Section~\ref{branched_section} is devoted to multivalued branched minimal immersions. Section~\ref{index_section} contains a new index bound for such immersions. In Section~\ref{Proof_section} we complete the proof of Theorem~\ref{YYth}. Finally, Theorem~\ref{Branching_th} is proved in Section~\ref{branching_sec}.

\subsection*{Acknowledgements} The author is grateful to V.~Baranovsky, M.~Coppens, A.~Neves, I.~Polterovich and R.~Schoen for fruitful discussions. A special thanks goes to M.~Coppens for providing the author with a copy of his thesis~\cite{Coppens_thesis}.  The author thanks V.~Medvedev and I.~Polterovich for remarks on the preliminary version of the manuscript.

\section{Preliminaries}
\label{prelim_section}
\subsection{Algebraic geometry: background and notations.} 
\label{AG_notations}
Let $(M,g)$ be an oriented Riemannian surface of genus $\gamma$. Isothermal coordinates induce the complex structure on $M$ compatible with the metric $g$. The surface $M$ endowed with this complex structure becomes a Riemann surface or a smooth complex curve. In the following holomorphic objects on $(M,g)$ are considered with respect to this particular complex structure. We remark that the complex structure only depends on the conformal class of the metric $g$.
The unit $2$-sphere $\mathbb{S}^2\subset\mathbb{R}^3$ admits the unique conformal class and as a result the unique complex structure. 

Since the conformal classes are closely related to complex structures, in the following we often make use of notations from the theory of Riemann surfaces, or equivalently, complex algebraic curves. One could consult~\cite[Chapter I]{ACGH} for a more detailed overview of the notions below. A divisor $D$ on a Riemann surface $M$ is a formal linear combination of points on $M$, $D = \sum_{p\in M} n_p p$, where $n_p\in\mathbb{Z}$ and $n_p\ne 0$ only for a finite set of points. The number $n_p$ is called {\em multiplicity} $\mult_p(D)$ of $D$ at $p$. The {\em degree} is defined as $\deg(D) = \sum n_p$. Divisor $D$ is {\em effective} if $\mult_p(D)\geqslant 0$ for all $p\in M$, we use notation $D\geqslant 0$ for effective divisors. To any meromorphic function $f$ on $M$ one associates a divisor $(f) = \sum\ord_p(f) p$, where $\ord_p(f) = k$ iff in the local complex coordinate centered at $p$ one has $f = z^k g(z)$, where $g$ is holomorphic and $g(p)\ne 0$. The divisor $(f)$ is represented as a difference of two effective divisors $(f) = N_f - P_f$, where 
$$
N_f =  \sum_{\ord_p(f)\geqslant 0}\ord_p(f) p\, ,\qquad P_f = -\sum_{\ord_p(f)\leqslant 0}\ord_p(f) p.
$$
$N_f$ is called {\em a null divisor} of $f$ and $P_f$ is called {\em a polar divisor} of $f$. Two divisors $D$ and $D'$ are {\em linearly equivalent} if $D-D' = (f)$ for some meromorphic function $f$. One sets
\begin{equation*}
\begin{split}
H^0(M,D) &= \{f\,|\,\text{$f$ is meromorphic},\, (f)+D\geqslant 0\} \\
h^0(D) &=\dim_{\mathbb{C}}H^0(M,D)
 \end{split}
\end{equation*}
and we identify $D$ with a line bundle whose local sections over $U\subset M$ are $H^0(U,D)$. This way, tensor product of bundles corresponds to sum of divisors, dual to the line bundle $D$ is a bundle corresponding to the divisor $-D$ and $\deg(D)$ is the first Chern number of the bundle. We use $K$ to denote the divisor of any meromorphic $1$-form, recall that $h^0(K) = \gamma$.

The {\em complete linear series} $|D|$ is the set of effective divisors linearly equivalent to $D$, i.e. it is a projectivization of $H^0(M,D)$. A {\em linear series} $\mathcal V$ is a projectivization $\mathbb{P}V$ of a linear subspace $V\subset H^0(M,D)$. 
The {\em base locus} of $\mathcal V$ is the divisor 
$$
B = \sum_{p}\min_{f\in V}(\ord_p(f))p.
$$
A series $\mathcal V$ is {\em base-point free} if $B = 0$.
A series $\mathcal V$ is called {\em a pencil} if $\dim V=2$. The classical (and essentially the only) examples of pencils come from holomorphic maps $\phi\colon M\to\mathbb{S}^2=\mathbb{CP}^1$ and are given by 
$$
\mathcal V_\phi = \left\{\sum_{q\in\phi^{-1}(p)}(\ord_q(d\phi)+1)q\right\}_{p\in\mathbb{CP}^1}
$$
The divisor corresponding to $p\in\mathbb{CP}^1$ is just the preimage $\phi^{-1}(p)$, where each point is taken with an appropriate multiplicity. Note that $\mathcal V_\phi$ is base-point free.
Conversely, given a base-point free pencil $\mathcal V$, let $f_0,f_1\in V\subset H^0(M,D)$ be a basis of $V$. Then $\phi\colon M\to\mathbb{CP}^1$ given in homogeneous coordinates by $\phi = [f_0\colon f_1]$ satisfies $\mathcal V_\phi=\mathcal V$.
In general, if the base locus $B\ne 0$, then $\phi$ can be defined over $B$ and $\mathcal V_\phi$ can be identified with a base-point free pencil $\mathcal V'\subset \mathbb{P}H^0(M,D-B)$.

Degree of a linear series is the degree of any divisor in it. 
If one identifies the holomorphic map  $\phi\colon M \to\mathbb{CP}^1$ with a meromorphic function, then one has  $\mathcal V_\phi\subset |P_\phi|$ and $\deg V_\phi=\deg\phi$.

%
%
%

\subsection{Conformally covariant Schr\"odinger operator}
Let $\phi\colon M\to\mathbb{S}^2$ be a holomorphic map and let $g$ be any smooth metric compatible with the complex structure. Montiel and Ros~\cite{MR} defined the operator $L_{\phi,g}\colon C^\infty(M)\to C^\infty(M)$ using the formula,
$$
L_{\phi,g}(u) = \Delta_gu-|\nabla\phi|_g^2u,
$$
where $|\nabla\phi|^2_g = \sum_{i=1}^3|\nabla\phi_i|_g^2$. If $\tilde g = f^2g$ is another metric compatible with the complex structure, then
$$
L_{\phi,\tilde g}(u) = f^{-2}L_{\phi,g}(u),
$$
i.e. $L_{g,\phi}$ is conformally covariant. As a result, the number of negative eigenvalues and the number of zero eigenvalues of $L_{\phi,g}$ do not depend on the choice of metric $g$. Indeed, the associated quadratic form
$$
Q_\phi(u) = \int\limits_M|\nabla u|_g^2 - |\nabla \phi|_g^2u^2\,dv_g
$$
is independent of the choice of $g$.

\begin{definition}
The index $\ind(\phi)$ of a holomorphic map $\phi\colon M\to\mathbb{S}^2$ is defined as a number of negative eigenvalues of the operator $L_{\phi,g}$ for some (any) choice of the metric $g$ compatible with the complex structure on $M$. 

The nullity $\nul(\phi)$ of $\phi$ is defined as a number of zero eigenvalues of $L_{\phi,g}$ for some (any) choice of $g$ compatible with the complex structure.
\end{definition}
\begin{remark}
Since any holomorphic map is harmonic, one has $\phi_i\in\ker(L_{\phi,g})$ for all $g$ and $i=1,2,3$. As a result, $\nul(\phi)\geqslant 3$.
\end{remark}


It is often useful to choose metrics with conical singularities. Let us describe those in more detail.

\subsection{Metrics with conical singularities.} 
As we have mentioned in the introduction, we consider an extended conformal class of metrics, namely, metrics with conical singularities. For our purposes it is enough to consider metrics with conical angles which are multiples of $2\pi$. These are very easy to describe. For a fixed metric $g$ such metrics have the form $h=f^2g$, where $f\geqslant 0$ is a smooth function with isolated zeroes. The set of zeroes $Z$ of $f$ corresponds to singular points of $h$. The metric $h$ induces a well-defined measure $dv_h = f^2\,dv_g$ on $M$. Moreover, by conformal invariance, the form $Q_\phi$ satisfies
$$
Q_\phi(u) =\int\limits_{M\setminus Z}|\nabla u|_h^2 - |\nabla \phi|_h^2u^2\,dv_h,
$$
where the right hand side is treated as an improper integral. We consider the Friedrichs extension of the operator $L_{\phi,h}$ on the space $L^2(dv_h)$ with domain $C^\infty_0(M\setminus Z)$. Since the capacity of a point is zero, the domain of this extension contains $C^\infty(M)$. Moreover, the eigenfunctions are smooth and the classical variational formula for eigenvalues holds. 

Having discussed metrics with conical singularities, we consider two choices of metrics $g$ that make $L_{\phi,g}$ have a particularly nice form.

\subsection{Induced metric}
\label{induced_section}
 Set $g = \phi^*g_{\mathbb{S}^2}$. Then $g$ is a metric with conical singularities at branch points of $\phi$ and $|\nabla\phi|^2_g\equiv 2$. We conclude that $\ind(\phi)$ is the number of eigenvalues of $\Delta_g$ that are less than $2$ and $\nul(\phi)$ is multiplicity of the eigenvalue $2$, i.e. using the Weyl's counting function notations from the introduction, one has $\ind(\phi)=N_{g}(2)$.

%

\subsection{Branched multivalued minimal immersions into $\mathbb{R}^3$} 
\label{branched_section}
Let $\widetilde M$ be the universal cover of $M$. The branched minimal immersion $X\colon \widetilde M\to\mathbb{R}^3$ is called {\em multivalued branched minimal immersion of $M$} if there exists a linear representation $\rho\colon\pi_1(M)\to\mathbb{R}^3$ such that for all $\sigma\in\pi_1(M)$, $p\in\widetilde M$ one has $X(\sigma\cdot p) = X(p) + \rho(\sigma)$.

Using parallel translations in $\mathbb{R}^3$ the pullback of $X^*T\mathbb{R}^3$ can be identified with $\widetilde M\times \mathbb{R}^3$. The differential $dX$ is a section of $\mathrm{Hom}(T\widetilde M,\widetilde M\times \mathbb{R}^3)$ with zeroes at branch points of $X$. 
The branching divisor $\widetilde B$ is defined so that for all points $p\in\widetilde M$ one has $\mult_p\widetilde B = \ord_p (dX)$. Points with $\mult_p\widetilde B\ne0$ are called branch points, otherwise they are called regular. The image of $dX$ at a regular point of $X$ is parallel to the tangent plane to the image of $X$. By~\cite{Osserman} the tangent plane can be smoothly extended to branch points. As a result the closure of the image of $dX$ is a $2$-dimensional subbundle $\widetilde L\subset \widetilde M\times \mathbb{R}^3$. The orientations on $\widetilde M$ and $\mathbb{R}^3$ uniquely define the unit section $\widetilde N$ of the orthogonal complement to $\widetilde L$. 
The induced metric $\widetilde g = X^*g_{\mathbb{R}^3}$ is a metric with conical singularities at branch points of $X$. At a regular point $p$ of $X$ one has $|\nabla\widetilde\phi|_{\widetilde g}^2(p) = -2K(p)$, where $K(p)$ is the Gauss curvature of $X(M)$ at the point $X(p)$. The metric $\widetilde g$ defines an associated complex structure on $\widetilde M$. With respect to this complex structure the Gauss map
 $\widetilde \phi\colon \widetilde M\to\mathbb{S}^2$ given by $p\mapsto \widetilde N(p)$ is holomorphic, see e.g.~\cite{Osserman}. 

 Moreover, since $X(\sigma\cdot p)$ is a parallel translation of $X(p)$ by $\rho(\sigma)$, 
 one has that $\widetilde L(p) = \widetilde L(\sigma\cdot p)$ as subspaces of $\mathbb{R}^3$.
 As a result, one defines a $2$-dimensional subbundle $L\subset M\times\mathbb{R}^3$. Similarly, all the objects defined via $dX$ descend to corresponding objects on $M$. This way we have a complex structure, a metric $g$ with conical singularities compatible with that structure, a branch divisor $B$ and the holomorphic Gauss map $\phi\colon M\to\mathbb{S}^2$, such that $|\nabla\phi|_g^2(p) = -2K(p)$. As a result, the operator $L_{\phi,g}$ takes the form 
 $$
 L_{\phi,g} (u) = \Delta_g u - 2K u,
 $$
which is the {\em Jacobi} (or {\em index}) {\em operator}. 
Negative eigenfunctions of this operator correspond to normal variations of $X$ in the class of branched multivalued immersions that decrease the area $\area_g(M)$. As a result $\ind(\phi)$ coincides with the index of the minimal immersion $X$ considered as a multivalued branched minimal immersion.

\subsection{Index bound} 
\label{index_section}
The idea behind the proof of Theorem~\ref{YYth} is to use the interplay between the two previous interpretations of $\ind(\phi)$. Namely, for a given $\phi$ we will construct a special $X$ so that the application of the following proposition yields Theorem~\ref{YYth}.
\begin{proposition}
\label{Immersionprop}
Let $X$ be a branched multivalued minimal immersion of $M$ to $\mathbb{R}^3$ with Gauss map $\phi$ and branching divisor $B$. Then 
$$
\ind(\phi)\geqslant \frac{2h^0(K-B)-3}{3}.
$$
\end{proposition}
\begin{remark}
In case $B=0$ this was proven by Ros, although he did not state it in this form, for the closest statement see e.g.~\cite[Theorem 14]{Ros}.
The proof below is an adaptation of his ideas to the branched setting.
\end{remark}
\begin{proof}[Proof of Proposition~\ref{Immersionprop}]
Let $\mathcal H(B)$ be a space of harmonic 1-forms on $M$ vanishing to the order of $B(p)$ at branch points $p$. One can identify $\mathcal H(B)$ with real parts of elements in $H^0(M,K-B)$, i.e. $\dim_{\mathbb{R}}\mathcal H(B) = 2h^0(K-B)$.

Let $\langle\cdot,\cdot\rangle$ be the metric on $1$-forms induced by $\widetilde g = X^*g_{\mathbb{R}^3}$. We claim that for any two elements $\omega_1,\omega_2\in \mathcal H(B)$ the inner product $\langle\omega_1,\omega_2\rangle$ is a smooth function on $M$. We only need to check this statement locally in the neighbourhood of branched points. Let $U$ be a neighbourhood of a branch point $p$ such that $\mult_p B = k$. Then in local complex coordinates centered at $p$ one has $\omega_i = \mathrm{Re}(z^k\xi_i(z)dz)$, $i=1,2$, where $\xi_i(z)$ is holomorphic, and $g = |z|^{2k}f^2(z)|dz|^2$, where $f(z)\ne 0$ in $U$. Then the direct computation yields $\langle\omega_1,\omega_2\rangle = f^{-2}(z)\mathrm{Re}(\xi_1\bar \xi_2)$, which is smooth in $U$.

It is well-known that components of the minimal immersion into $\mathbb{R}^3$ are harmonic functions. Since the differential commutes with the Laplacian, the components of the differential $dX$ are harmonic $1$-forms on $M$. Moreover, they have zeroes of order $\mult_p B$ at $p$, therefore they lie in $\mathcal H(B)$.
For each element $\omega\in\mathcal H(B)$ consider the vector field 
$$
V_\omega=(\langle dX^1,\omega\rangle, \langle dX^2,\omega\rangle, \langle dX^3,\omega\rangle),
$$
where $dX^i$, $i=1,\ldots,3$ are components of $dX$.
Outside branch points, $V_\omega$ is a section of $L$. Thus, by our previous discussion
$V_\omega$ is a smooth vector function on $M$ satisfying $\sum_{i=1}^3 N^i V_\omega^i = 0$, where $N^i$ are components of the normal field $N$.

Ros~\cite[Lemma 1]{Ros} proved the following proposition.
\begin{proposition} 
\label{Ros_proposition}
At the regular points  one has
$$
\Delta V_\omega +2KV_\omega = -2\langle \nabla\omega, II\rangle N,
$$
where $\Delta$ is the Laplacian on functions applied component-wise and $II$ is the second fundamental form.
Moreover, the space $\mathcal L = \{\omega\in\mathcal H(B)|\,\, \langle\nabla\omega,II\rangle = 0\}$ is of dimension at most $3$.
\end{proposition}
\begin{remark}
In~\cite[Lemma 1]{Ros} Ros proved this proposition for minimal surfaces without branched points. In particular, he shows that the space of $\omega$ such that $\langle\nabla\omega,II\rangle = 0$ is exactly three-dimensional. The proof is purely local and, thus, yields the statement of Proposition~\ref{Ros_proposition}. The reason the dimension of $\mathcal L$ might be less than $3$ is that there is no control over the behaviour of $\omega$ in the neighbourhood of the branched points, i.e. the harmonic forms from the argument of Ros might not lie in $\mathcal H(B)$.
\end{remark}

The idea now is to use the components of $V_\omega$, $\omega\in\mathcal H(B)$ as test functions for the quadratic form $Q_\phi$. Assume the contrary to the statement of Proposition~\ref{Immersionprop}, i.e. that $3\ind(\phi)<\dim\mathcal H(B)-3$. Let $U^-_\phi$ be the negative space of $Q_\phi$, $\dim U^-_\phi=\ind(\phi)$. Then by our assumption there exists $\omega\in\mathcal H(B)\backslash\mathcal L$ such that $V_\omega^i\perp U^-_\phi$ for $i=1,2,3$.
Then one has 
$$
\sum_{i=1}^3Q_\phi(V_\omega^i,V_\omega^i)\geqslant 0. 
$$
Let $\psi_\varepsilon$ be a logarithmic cut-off function which is equal to $1$ outside an $\varepsilon$-small neighbourhood of branch points and $\int|\nabla\psi_\varepsilon|_{g_0}^2\,dv_{g_0}\to 0$ as $\varepsilon\to 0$.
Since $V^i_\omega$ are smooth one has $Q_\phi(V_\omega^i,\psi_\varepsilon V_\omega^i)\to Q_\phi(V_\omega^i,V_\omega^i)$ as $\varepsilon\to 0$. Then by Proposition~\ref{Ros_proposition} one has
\begin{equation*}
\begin{split}
\sum_{i=1}^3Q_\phi(V_\omega^i,\psi_\varepsilon V_\omega^i) =& \int\limits_M\sum_{i=1}^3 (\Delta_gV^i_\omega + 2KV^i_\omega)\psi_\varepsilon V^i_\omega\,dv_g = \\&-2\int\limits_M\langle\nabla\omega,II\rangle\sum_{i=1}^3\psi_\varepsilon N^i V^i_\omega\,dv_g = 0,
\end{split}
\end{equation*}
since $\sum_{i=1}^3 N^i V_\omega^i = 0$. Passing to the limit $\varepsilon\to 0$ we conclude that $V_\omega^i\in\ker L_{g,\phi}$ which is a contradiction, since $\omega\not\in\mathcal L$.
%
%
\end{proof}

\section{Proof of Theorem~\ref{YYth}}
\label{Proof_section}
In this section we make use of Proposition~\ref{Immersionprop} in order to prove the following proposition.
\begin{proposition}
\label{S2th}
Let $\phi\colon M\to\mathbb{S}^2$ be a holomorphic map of degree $d_\gamma = \left[\frac{\gamma+3}{2}\right]$. If $\gamma>2$ then $\ind(\phi)>1$.
\end{proposition}

\subsection{Laplace eigenvalues.}
\label{prelim}
In this section we take the viewpoint of Section~\ref{induced_section} and identify $\ind(\phi)$ with the number $N_{\phi^*g_{\mathbb{S}^2}}(2)$.
This way, Theorem~\ref{YYth} is an easy consequence of Proposition~\ref{S2th}.
\begin{proof}[Proof of Theorem~\ref{YYth}]
For any Riemann surface $M$ there exists a holomorphic map $\phi\colon M\to\mathbb{S}^2$ of degree at most $d_\gamma$, see~\cite{GH}. By~\cite[Proposition 3.3]{CKM} (see also~\cite{ESI, LiYau}) for any metric $g$ compatible with the complex structure one has 
\begin{equation}
\label{equality:CKM}
\bar\lambda_1(M,g)\leqslant 8\pi \deg(\phi);
\end{equation}
the equality holds iff there exists a conformal automorphism $\sigma$ of $\mathbb{S}^2$ such that 
$g$ is homothetic to $(\sigma\circ\phi)^*g_{\mathbb{S}^2}$ and the components of $\sigma\circ\phi$ are the first eigenfunctions of $g$. Since $\deg(\phi)\leqslant d_\gamma$, the equality in Yang-Yau inequality~\eqref{YYineq} occurs iff 
 $\deg(\phi)=d_\gamma$ and the inequality~\eqref{equality:CKM} is an equality.
The latter implies that $\ind(\sigma\circ\phi) = 1$, which in the case $\gamma>2$ contradicts Proposition~\ref{S2th} since $\deg(\sigma\circ\phi) = \deg(\phi)=d_\gamma$. 
\end{proof}

The rest of this section is devoted to the proof of Proposition~\ref{S2th}. 
We start with the following theorem. It was proved by Montiel and Ros in~\cite{MR}. A more direct proof appears in~\cite[Proposition 3.4]{CKM}.
\begin{theorem}
\label{uniquepencil}
Let $\phi\colon M\to\mathbb{S}^2$ be a holomorphic map of index $1$. Then for any other holomorphic map $\psi\colon M\to\mathbb{S}^2$ one has $\deg(\psi)\geqslant\deg(\phi)$. Moreover, $\deg(\psi)=\deg(\phi)$ iff there exists a conformal automorphism $\sigma$ of $\mathbb{S}^2$ such that $\psi = \sigma\circ \phi$.
\end{theorem}
 
Using the language of algebraic geometry introduced in Section~\ref{AG_notations}, the previous theorem takes the following form.
 
\begin{theorem}
\label{ind1}
Let $\phi$ be a meromorphic function on $M$ and let $\mathcal V_\phi$ be the associated pencil. If $\ind(\phi) = 1$, then $\mathcal V_\phi$ is the unique pencil on $M$ of degree $\deg(\phi)$ and there are no pencils on $M$ of smaller degree. 
 \end{theorem}
 \begin{remark}
 Note that the pencil of minimal degree is automatically a base-point free complete linear series, i.e. if $\deg(\phi)$ is a minimal possible degree of a meromorphic function on $M$ then $\mathcal V_\phi = |P_\phi|$.
 \end{remark}
 
Brill-Noether theory~\cite{GH} is concerned with questions related to existence of linear series on a Riemann surface $M$. Let us recall the statements of this theory related to pencils and discuss their consequences. We refer to~\cite{KleimanLaksov} for the particular case of pencils. Note that a similar discussion appears in~\cite[Section 2.2]{Ros}.
 \begin{itemize}
\item[1)] For any Riemann surface $M$ of genus $\gamma$ there exists a pencil of degree at most $d_\gamma$. Therefore, any $\phi$ with $\deg(\phi)>d_\gamma$ is of index at least $2$.
\item[2)] If genus $\gamma$ is odd then there is at least one-dimensional family of pencils of degree $d_\gamma$. Therefore, Proposition~\ref{S2th} holds for odd genera.
\item[3)] If genus $\gamma$ is even, then a generic Riemann surface of genus $\gamma$ has exactly $N_\gamma$ pencils of degree $d_\gamma$. Moreover, one has $N_2=1$ and $N_\gamma>1$ for $\gamma>2$. Therefore, by Theorem~\ref{uniquepencil} one has that Proposition~\ref{S2th} holds for {\em generic} Riemann surface of even genus. Generic here means non-empty Zariski open subset of moduli space of Riemann surfaces of genus $\gamma$, or, equivalently, the complement to an analytic set of codimension at least $1$.
 \end{itemize}
 
 To summarize, we obtain that the classical Brill-Noether theory in combination with Proposition~\ref{uniquepencil} imply Proposition~\ref{S2th} in the following two cases: 1) for {\em any} Riemann surface of odd genus; 2) for {\em generic} Riemann surfaces of even genus $\gamma>2$.  In the remainder of this section we deal with non-generic complex structures on surfaces of even genus $\gamma$. 
%
\subsection{Non-generic Riemann surfaces}
\label{Nongeneric}
The following proposition was proved in~\cite{Coppens_thesis}. A sketch of the argument appears in the appendix to~\cite{Coppens1}.
  \begin{proposition}
 \label{AGprop}
 Let $|D|$ be a pencil. Then $h^0(2D)>3$ iff $|D|$ is a limit of two different pencils on the moduli space of curves.
 \end{proposition}
 \begin{remark}
For clarity of exposition, Proposition~\ref{AGprop} is stated rather informally. The rigorous form of this statement can be found in~\cite[p. 65]{Coppens_thesis}. 
For the sake of understanding the statement of Proposition~\ref{AGprop}, one could think of "a limit of two different pencils" as the fact that there exists a holomorphic family $\Sigma_t$ of Riemann surfaces parametrized by $t\in \mathbb{D}\subset C$, such that $\Sigma_0 = \Sigma$, and two families $\mathcal V^t_1\not\equiv\mathcal  V^t_2$ of pencils on $\Sigma_t$ such that $|D|=\mathcal V^0_1=\mathcal  V^0_2$. 
 \end{remark}
 \begin{remark}
 Below, we only make use of "if"-part of Proposition~\ref{AGprop}. As it is remarked in~\cite{Coppens1,Coppens_thesis}, the "if"-part easily follows from the semi-continuity theorem~\cite{Hartshorne}. The following informal argument clarifies the main idea. Assume $\phi_t\in H^0(\Sigma_t, D_1^t)$, where $D_1^t\in\mathcal V^t_1$ and, similarly, $\psi_t\in H^0(\Sigma_t, D_2^t)$, where $D_2^t\in\mathcal V^t_2$. Consider $F_t=D_1^t+D_2^t$. On one hand, $|F_0|=|2D|$. On the other hand, $\mathrm{span}\{1,\phi_t,\psi_t,\phi_t+\psi_t\}\subset H^0(\Sigma_t, F_t)$, i.e. $h^0(\Sigma_t,F_t)>3$. The semi-continuity theorem asserts that $h^0(\Sigma_t,F_t)$ is an upper semi-continuous function of $t$ and, therefore, $h^0(\Sigma, 2D) = h^0(\Sigma_0, 2F_0)>3$. 
 \end{remark}
 As a corollary we obtain the following statement.
 \begin{corollary}
 \label{nongeneric:cor}
 Suppose that $M$ is a Riemann surface of genus $\gamma>2$ that has a unique pencil $|D|$ of degree $d_\gamma$ and no pencils of smaller degree. Then $h^0(2D)>3$.
 \end{corollary}
 \begin{proof}
%
For a Riemann surface $\Sigma$, let $\mathcal G_{d_\gamma}^1(\Sigma)$ be the space of all pencils of degree $d_\gamma$ on $\Sigma$. By the upper semi-continuity of dimension, see~\cite{Hartshorne}, $U = \{\Sigma\in\mathcal M_\gamma\,|\, \dim\mathcal G_{d_\gamma}^1(\Sigma) = 0\}$ is a Zariski-open subset of $\mathcal M_\gamma$. By Brill-Noether theory, see e.g.~\cite[Chapter XXI]{ACGH2}, $U$ is non-empty. Let $\mathcal G_{d_\gamma}^1(U)=\{(\Sigma,\mathcal V)|\,\Sigma\in U,\,\mathcal V\in \mathcal G_{d_\gamma}^1(\Sigma)\}$ be the space of pencils on $U$ and let $\pi\colon \mathcal G_{d_\gamma}^1(U)\to U$ be the natural projection. Furthermore, by~\cite{KleimanLaksov} $V =\{\Sigma\in U\,|\, |\mathcal G_{d_\gamma}^1(\Sigma)| = N_\gamma\}$ is an open subset of $U$, and by~\cite[Chapter XXI, Proposition 6.8]{ACGH2} one has that $\mathcal G_{d_\gamma}^1(U)$ is smooth. Therefore, $\pi$ is a branched covering with $N_\gamma>1$ leaves and ramification locus $U\setminus V$. Since $M\in U\setminus V$, one has that there exists a deformation $M_t$ of $M$ and two pencils $\mathcal V^t_1 \ne \mathcal V^t_2$ on $M_t$ such that the limits of both $(M_t,\mathcal V^t_1)$ and $(M_t,\mathcal V^t_2)$ in $\mathcal G_{d_\gamma}^1(U)$ are equal to $(M,|D|)$. The application of Proposition~\ref{AGprop} completes the proof.
 \end{proof}
 \begin{remark}
 We remark that if $\gamma$ is odd and $D$ is any divisor of degree $d_\gamma$, then by Riemann-Roch theorem one has $h^0(2D)>3$.
 \end{remark}
 
 \subsection{Proof of Proposition~\ref{S2th}}
 In this section we complete the proof of Proposition~\ref{S2th} by establishing the conclusion for non-generic complex structures on surfaces of even genus $\gamma>2$, i.e. we assume that $\mathcal V_\phi = |P_\phi|$ is the unique pencil of degree $d_\gamma$ on $M$ and there are no pencils of smaller degree.
 
 By Corollary~\ref{nongeneric:cor} one has $h^0(2P_\phi)>3$.
By Riemann-Roch theorem one has
 $$
 h^0(K-2P_\phi) = \gamma - 1 - 2\deg(\phi) + h^0(2P_\phi)
 $$
 At the same time, $\gamma$ is even and, thus, $2\deg\phi\leqslant \gamma+2$. As a result, one obtains
 $$
  h^0(K-2P_\phi) >\gamma-1 - (\gamma+2) + 3=0.
 $$
 Let $\omega$ be a non-constant element in $H^0(K-2P_\phi)$ viewed as a holomorphic form with zeroes at poles of $\phi$ such that $\ord_p(\omega)\geqslant 2\mult_pP_\phi$. 
 
At this point we use the Weierstrass representation of minimal surfaces in $\mathbb{R}^3$, see e.~g.~\cite{Osserman}. Namely, we set 
 $$
X(p) = \int^p_{p_0}\mathrm{Re}\left((1+\phi^2)\omega, i(1-\phi^2)\omega, 2\phi\omega\right)
 $$
 to be the multivalued branched minimal immersion of $M$ with Gauss map $\phi$. The branch points of $X$ correspond to points where all three components of the integrand are zero, i.e. to points $p$, where $\ord_p(\omega)> -2\ord_p(\phi)$. Therefore, the branching divisor $B$ of $X$ satisfies $B=K-2P_\phi$. Finally, by Proposition~\ref{Immersionprop} one has
 $$
 \ind(\phi)\geqslant\frac{2h^0(K-(K-2P_\phi))-3}{3} = \frac{2h^0(2P_{\phi})-3}{3}>1
 $$
 by the assumption that $h^0(2P_{\phi})>3$. This completes the proof of Proposition~\ref{S2th}.

\section{Upper bounds on total branching order}
\label{branching_sec}
In this section we investigate the total branching order of harmonic maps to the sphere. In particular, we prove Theorem~\ref{Branching_th}.
\subsection{Harmonic sequence} In the present section we discuss the concept of a {\em harmonic sequence} introduced by Chern and Wolfson in~\cite{CW1,CW2}. We follow the exposition in~\cite{BJRW}.
 
Let $L\subset\mathbb{CP}^n\times\mathbb{C}^{n+1}$ be a tautological bundle over $\mathbb{CP}^n$, i.e $L = \{(l,v)\,|v\in l\}$. Let $M$ be a Riemann surface. There is a correspondence between smooth maps $\psi\colon M\to\mathbb{CP}^n$ and line subbundles of a trivial bundle $M\times \mathbb{C}^{n+1}\to M$ given by $\psi\leftrightarrow\psi^*L$. We endow $M\times\mathbb{C}^{n+1}$ with the usual Hermitian inner product $\langle\cdot,\cdot\rangle$ and the induced Hermitian connection.
Then by Koszul-Malgrange theorem all line subbundles of $M\times\mathbb{C}^{n+1}$ are automatically holomorphic. Moreover, one has the bundle isomorphism $T\mathbb{CP}^n\cong \Hom(L,L^\perp)$. We endow all these bundles with Hermitian connections induced from the Hermitian connection on the trivial $\mathbb{C}$-bundle.
Then differential $d\psi$ could be considered as an element of $\Hom(TM\otimes\psi^*L,\psi^*L^\perp)$ and the $(1,0)$-part of the complexified differential map defines 
\begin{equation}
\label{partial}
\partial\colon T^{(1,0)}M\otimes\psi^*L\to\psi^*L^\perp,
\end{equation}
and the $(0,1)$-part of the complexified differential map defines 
\begin{equation}
\label{barpartial}
\bar\partial\colon  T^{(0,1)}M\otimes\psi^*L\to\psi^*L^\perp
\end{equation}

Let $g$ be any metric compatible with the complex structure on $M$. Assume that $\psi\colon (M,g)\to\mathbb{CP}^n$ is a linearly full (i.e. its image is not contained in the projective subspace) harmonic map, where $\mathbb{CP}^n$ is endowed with the Fubini-Study metric. In local complex coordinates the harmonicity can be expressed as $(\nabla d\psi)(\partial_{\bar z},\partial_z)=(\nabla d\psi)(\partial_{z},\partial_{\bar z})=0$, which is equivalent to the fact that $\partial$ ($\bar\partial$) defined in~\eqref{partial} (in~\eqref{barpartial}) is a(n) (anti-)holomorphic morphism of bundles. Thus their images can be defined across zeroes of $d\psi$ and give rise to line subbundles $L_1$ ($L_{-1}$). Denoting $\psi^*L$ by $L_0$ we have a holomorphic map
$$
\partial_0\colon T^{(1,0)}M\otimes L_0\to L_1
$$
and an antiholomorphic map
$$
\bar\partial_0\colon T^{(0,1)}M\otimes L_0\to L_{-1}
$$

Bundles $L_1$ and $L_{-1}$ correspond to maps $\psi_1,\psi_{-1}\colon M\to\mathbb{CP}^n$. It is proved in~\cite{CW1} that if $\psi_0=\psi$ is harmonic then so are $\psi_1$ and $\psi_{-1}$. Repeating the process one constructs a sequence of bundles $\{L_p\}$, holomorphic maps
$$
\partial_p\colon T^{(1,0)}M\otimes L_p\to L_{p+1}
$$
and antiholomorphic maps
$$
\bar\partial_p\colon T^{(0,1)}M\otimes L_p\to L_{p-1}.
$$
This collection of data is referred to as {\em a harmonic sequence associated to $\psi = \psi_0$}.

If $\partial_p\equiv0$ ($\bar\partial_p\equiv0$) but $\partial_{p-1}\not\equiv0$ ($\bar\partial_{p-1}\not\equiv 0$), then we say that the harmonic sequence terminates with $L_p$ at the right (left). In this case the map $\psi_p$ is antiholomorphic (holomorphic) and the harmonic sequence coincides with its Frenet frame.

For each bundle $L_p$ one can define a smooth bilinear form $\mathcal H_p$ on $L_p$ by restricting $\mathbb{C}$-bilinear dot product from $\mathbb{C}^{n+1}$ to $L_p$. We will use the notation $u\cdot v$ for this product. One can look at $\mathcal H_p$ as a generalization of the famous Hopf differential.

\begin{proposition}
\label{holo_prop}
If $L_{p-1}\perp \bar L_{p}$, then $\mathcal H_p$ is holomorphic section of $(L_p^*)^2$. 
\end{proposition}
\begin{proof}
One needs to check that $\nabla_{\partial_{\bar z}}\mathcal H_p=0$, where $\nabla$ is the induced connection on $(L_p^*)^2$. Let $\xi$ and $\zeta$ be local sections of $L_p$. 
For a subbundle $V\subset M\times\mathbb{C}^{n+1}$ let $P_V$ denote the orthogonal projection onto $V$.
Then one has
\begin{equation*}
\begin{split}
&(\nabla_{\partial_{\bar z}}\mathcal H_p)(\xi,\zeta) = \partial_{\bar z}(\xi\cdot\zeta) - P_{L_p}(\partial_{\bar z}\xi)\cdot\zeta - \xi\cdot P_{L_p}(\partial_{\bar z}\zeta) \\
&= P_{L^\perp_p}(\partial_{\bar z}\xi)\cdot\zeta + \xi\cdot P_{L^\perp_p}(\partial_{\bar z}\zeta) = \bar\partial_p(\partial_{\bar z}\otimes\xi)\cdot\zeta + \xi\cdot\bar\partial_p(\partial_{\bar z}\otimes\zeta)
\end{split}
\end{equation*}
Since the image of $\bar\partial_p$ is $L_{p-1}$, the condition $L_{p-1}\perp \bar L_{p}$ guarantees that the right hand side vanishes.
\end{proof}

The following are two properties of $\partial_p$ that will be of use in the next section.
\begin{itemize}
\item[1)] The singular points of $\partial_p$ are precisely branching points of $\psi_p$.
\item[2)] The map $\partial_p$ is a holomorphic section of $K\otimes L_p^*\otimes L_{p+1}$ and therefore one has
\begin{equation}
\label{ramification}
r(\partial_p) = c_1(K\otimes L_p^*\otimes L_{p+1}) = 2\gamma-2+c_1(L_{p+1}) - c_1(L_p),
\end{equation}
where $r(\partial_p)\geqslant 0$ is the degree of the zero divisor of $\partial_p$ and is referred to as  {\em ramification index}. It coincides with the total branching order of $\psi_p$. 
\end{itemize}

\subsection{Harmonic sequences for maps to a sphere}
\label{HSS2}
In this section we specify the previous discussion to a particular case of a linearly full harmonic map $\Phi\colon M\to\mathbb{S}^n$. Let $\pi$ be a projection $\pi\colon\mathbb{S}^n\to\mathbb{RP}^n$ and $i$ be an embedding $i\colon\mathbb{RP}^n\to\mathbb{CP}^n$. Since $i$ is totally geodesic, the composition $\psi=i\circ\pi\circ\Phi$ is harmonic. Moreover, $\Phi$ is linearly full iff $\psi$ is linearly full. Let $\{L_i\}$ be the harmonic sequence associated to $\psi$. We note the following.
\begin{itemize}
\item[1)] One has $\langle\Phi,\Phi\rangle = 1$. Therefore, $\langle\partial_z\Phi,\Phi\rangle = \langle\partial_{\bar z}\Phi,\Phi\rangle = 0$, i.e. $\Phi$ is parallel.
\item[2)] $\Phi\colon M\to\mathbb{S}^n\subset\mathbb{R}^{n+1}\subset\mathbb{C}^{n+1}$ is a global nowhere zero section of $L_0$. Since $\Phi$ is parallel, $L_0$ is trivial and $c_1(L_0)=0$.
\item[3)] $\bar L_p = L_{-p}$.
\end{itemize}
 
By definition, $L_0\perp L_{-1} $, hence, by item $3)$ one has $L_0\perp \bar L_1$. Therefore, by Proposition~\ref{holo_prop}, the form $\mathcal H_1$ is a holomorphic bilinear form on $L_1$, i.e. an element of $H^0(M,(L_1^*)^2)$. Consider two cases.

{\bf Case 1.} $\mathcal H_1\not\equiv 0$. Then $h^0((L_1^*)^2)>0$ and as a result $c_1(L_1) \leqslant 0$.

{\bf Case 2.} $\mathcal H_1\equiv 0$. We claim that in this case $L_1\perp\bar L_2$, i.e. $\mathcal H_2$ is an element of $H^0(M,(L_2^*)^2)$. Indeed, let $\xi$ be a local section of $L_1$. Then, since $\mathcal H_1\equiv 0$, one has $\xi\cdot\xi = 0$. Application of $\partial_z$ to both sides of the equality yields
\begin{equation}
\label{perp_eq}
0 = 2(\partial_z\xi)\cdot\xi = 2(\partial_1(\partial_{ z}\otimes\xi)+\nabla^{L_1}_z\xi)\cdot\xi = 2\partial_1(\partial_{ z}\otimes\xi)\cdot\xi,
\end{equation} 
where in the last equality we once again used that $\mathcal H_1\equiv 0$. Since $\partial_1\xi$ is a local section of $L_2$, this concludes the proof of the claim.

Now we repeat the process for $\mathcal H_2$: either it is a non-zero element of $H^0(M,(L_2^*)^2)$ or $\mathcal H_2\equiv 0$. The following proposition makes this inductive process rigorous.
\begin{proposition}
\label{induction_prop}
Suppose that $\mathcal H_i\equiv 0$ for $i=1,\ldots,p$. Then the bundles $L_{-p},\ldots,L_p,L_{p+1}$ are mutually orthogonal.
\end{proposition}
\begin{proof}
We prove it by induction on $p$. The case $p=0$ follows from the definition. Assume that the statement is true for $p-1$, let us prove it for $p$. 

By the induction step one has that $L_{-p+1},\ldots,L_{p-1}, L_p$ are mutually orthogonal. By taking conjugates, one concludes $L_{-p},L_{-p+1},\ldots,L_{p-1}$ are mutually orthogonal. Furthermore, $L_p\perp \bar L_p = L_{-p}$ is equivalent to $\mathcal H_p = 0$, therefore $L_{-p},L_{-p+1},\ldots,L_{p-1}, L_p$ are mutually orthogonal. It remains to prove that $L_{p+1}\perp L_{i}$ for $|i|\leqslant p$. We do that in two steps.
First, since $\mathcal H_p=0$, repeating the argument of equality~\eqref{perp_eq} for $x\in L_{p}$, one obtains $L_{-p-1}=\bar L_{p+1}\perp L_{p}$. 
At the same time, by definition $L_{-p-1}\perp L_{-p}$. 
Taking conjugates yields $L_{p+1}\perp L_{-p}$ and $L_{p+1}\perp L_{p}$.
Second, let us show $L_{p+1}\perp L_i$ for $|i|\leqslant (p-1)$. Let $\xi$ and $\zeta$ be local sections of $L_p$ and $L_i$ respectively, where $|i|\leqslant p-1$. By induction step, one has $\xi\cdot\zeta = 0$. Similarly to~\eqref{perp_eq}, application of $\partial_z$ yields
$$
0 = \partial_z\xi\cdot\zeta + \xi\cdot\partial_z\zeta = \partial_p(\partial_{z}\otimes\xi)\cdot\zeta,
$$
where we used that $\partial_z\zeta$ is local section of $L_i\oplus L_{i+1}\perp L_{p}$ by the induction step. Since $\partial_p(\partial_{\bar z}\otimes\xi)$ is section of $L_{p+1}$, this completes the proof.
\end{proof}

Thus, the following process is well-defined. Take the largest $q$ such that $\mathcal H_i = 0$ for $1\leqslant i\leqslant q-1$ and $\mathcal H_q\not\equiv 0$. By Proposition~\ref{induction_prop} either such $q$ exists and satisfies $2q\leqslant n+1$ or $\mathcal H_i\equiv 0$ for all $i\geqslant 1$. As a result, we have one of the following cases.
\begin{itemize}
\item[(A)] There exists $q>0$ such that $\mathcal H_i = 0$ for $1\leqslant i\leqslant q-1$ and $\mathcal H_q\not\equiv 0$. Then, by Proposition~\ref{induction_prop} one has $\bar L_q\perp L_{q-1}$ and Proposition~\ref{holo_prop} implies that $\mathcal H_q\in H^0(M,(L_q^*)^2)$ is a non-zero section. Therefore, $c_1(L_q)\leqslant 0$ and $2q\leqslant n+1$.
\item[(B)] For all $p>0$, one has $\mathcal H_p\equiv 0$. Then by Proposition~\ref{induction_prop} one has that the harmonic sequence $\{L_i\}$ terminates at the right and at the left and $n=2m$. In particular, the harmonic sequence $\{L_i\}$ coincides with the Frenet frame of a holomorphic curve $\psi_{-m}$. In this case $\Phi$ and $\psi$ are both referred to as {\em totally isotropic}.
\end{itemize}

We treat these cases separately.
\subsection{Case (A)} 
\label{CaseA:sec}
Let us sum up the identities~\eqref{ramification} for $p=0,\ldots,q-1$. Keeping in mind that $r(\partial_0) = |B|$ and $c_1(L_0) = 0$ one obtains
$$
\deg(B)\leqslant \deg(B) + \sum_{p=1}^{p=q-1}r(\partial_p) = q(2\gamma-2) + c_1(L_q)\leqslant q(2\gamma-2)\leqslant (n+1)(\gamma-1). 
$$

\begin{proof}[Proof of Proposition~\ref{nonconformal:prop}] Recall that $\Phi$ is conformal iff $\mathcal H_1\equiv 0$. If $\Phi$ is not conformal, then $\mathcal H_1\not\equiv 0$ and, as a result, the previous inequality holds with $q=1$.
\end{proof}

\subsection{Case (B)} In this case the harmonic sequence coincides with the Frenet frame of the holomorphic curve $\psi_{-m}$ which is referred to as {\em directrix} of $\Phi$. One concludes that $n=2m$ and the harmonic sequence terminates with $L_m$ at the right and with $L_{-m}$ at the left.  The harmonic map $\Phi$ in this case is called {\em totally isotropic}. They were first studied in detail in~\cite{Barbosa}. In contrast to case (A) there is no upper bound for $\deg(B)$ in purely topological terms. 
Indeed, all harmonic maps with domain $M\cong \mathbb{S}^2$ are totally isotropic. Let $\pi_k\colon\mathbb{S}^2\to\mathbb{S}^2$ be a branched cover given by $z\mapsto z^k$. Then the maps $\Phi_k = \pi_k\circ \Phi$ are harmonic and have arbitrary large total branching order.
 
However, in this example the energy $E_g(\Phi_k)$ given by
$$
E_g(\Phi_k) = \frac{1}{2}\int_{\mathbb{S}^2}|\nabla\Phi_k|^2_g\,dv_g,
$$
grows with $k$. In fact, one has the following proposition.

\begin{proposition}
Suppose that $\Phi\colon (M,g)\to\mathbb{S}^{2m}$ is a linearly full totally isotropic harmonic map. Then one has
$$
\area_{(\Phi^*g_{\mathbb{S}^{2m}})}(M) = E_g(\Phi) = 2\pi\left(m(m+1)(1-\gamma) + \sum_{j=0}^{m-1}(m-j)r(\partial_j)\right).
$$
\end{proposition}
\begin{proof}
We note that totally isotropic harmonic maps are automatically conformal, i.e. $\Phi\colon (M,\Phi^*g_{\mathbb{S}^{2m}})\to\mathbb{S}^{2m}$ is a branched minimal immersion. Indeed, conformality is equivalent to $\mathcal H_1\equiv 0$.

This proposition easily follows from the known results on totally isotropic minimal immersions. Unfortunately, we were not able to find the exact reference, so we sketch a proof here. We follow the paper~\cite{BJRW} and remark that all computations in that paper are valid for any totally isotropic immersion, not necessarily for the ones with domain diffeomorphic to $\mathbb{S}^2$. The only correction one should make is that instead of $c_1(T^{(1,0)}M) = -2$ one has $c_1(T^{(1,0)}M) = 2\gamma-2$. Similarly, one could follow~\cite[Section 6]{Barbosa}.

Recall that $\psi_{-m}$ is a holomorphic map $\psi_{-m}\colon M\to\mathbb{CP}^{2m}$. Let $\delta_k$ be the degree of its $k$-th osculating curve $\sigma_k\colon M\to \mathbb{CP}^{{2m+1\choose k+1} -1}$. Note that $\delta_0 = \deg(\psi_{-m})=-c_1(L_{-m})$. Moreover, computations in~\cite[Section 3]{BJRW} show that
\begin{equation}
\label{osculating}
\begin{split}
c_1(L_{-m+k}) &= \delta_{k-1}-\delta_k, \\
E_g(\psi_{-m+k}) &= \pi(\delta_{k-1}+\delta_k),
\end{split}
\end{equation}
where we assume $\delta_{-1}=0$.
Recall that that $c_1(L_0) = 0$ which implies $E_g(\Phi) = E_g(\psi_0) = 2\pi\delta_m$. Summing up equalities~\eqref{osculating} for $k=0,\ldots,m-1$ one obtains
$$
\delta_m = -\sum_{p=-m}^{0} c_1(L_p) = \sum_{p=1}^m c_1(L_p),
$$
where we used that $\sum c_1(L_i) = c_1(\oplus L_i) = 0$.
Finally, we apply relation~\eqref{ramification} to express the right hand side of the previous equality in terms of ramification indices. Namely, summing up relations~\eqref{ramification} for $p=0,\ldots,k-1$ and taking into account that $c_1(L_0) = 0$ one obtains
$$
c_1(L_k) + k(2\gamma-2) = \sum_{p=0}^{k-1} r(\partial_p).
$$
Summing these equalities for $k=1,\ldots,m$ one obtains
$$
E(\Phi) = 2\pi\delta_m = 2\pi\left(\sum_{k=1}^m\sum_{p=0}^{k-1} r(\partial_p) + 2(1-\gamma)\sum_{k=1}^mk\right).
$$
Rearranging the terms yields the proposition.
\end{proof}

Recalling that $r(\partial_0) = \deg(B)$ one obtains the following corollary.
\begin{corollary} Let $\Phi\colon (M,g)\to\mathbb{S}^{2m}$ be a linearly full totally isotropic harmonic map. Then the total branching $\deg(B)$ satisfies
$$
\deg(B)\leqslant \frac{1}{2\pi m}E_g(\Phi) +(m+1)(\gamma-1).
$$ 
\end{corollary}

This completes the proof of Theorem~\ref{Branching_th} for orientable surfaces $M$. Assume $M$ is non-orientable and $\Phi\colon (M,g)\to\mathbb{S}^n$ is a harmonic map with total branching $|B|$. Let $\pi\colon \widetilde M\to M$ be an orientable double cover, $\widetilde\Phi = \Phi\circ\pi$ be a harmonic map $(\widetilde M,\pi^* g)\to\mathbb{S}^n$ with total branching order $\deg(\widetilde B)$. Applying previous arguments to $\widetilde \Phi$, noting that $\deg(\widetilde B) = 2\deg (B)$, $\chi(\widetilde M) = 2\chi(M)$ and $E_{\pi^*g}(\widetilde \Phi) = 2E_g(\Phi)$ completes the proof in the non-orientable case.

\subsection{Proof of Proposition~\ref{extremalbranching_prop}} 
\label{extremalbranching_sec}
Recall that Proposition~\ref{extremalbranching_prop} states that the total number of conical singularities of any $\bar\lambda_k$ conformally extremal metric is bounded by $C(\gamma+1)(\gamma+k)$ for some universal constant $C$.

\begin{proof}
By Theorem~\ref{extremal_th} there exists a harmonic map $\Psi\colon(M,h)\to\mathbb{S}^n$, whose components are the $k$-th eigenfunctions. According to the discussion before Proposition~\ref{extremalbranching_prop}, it is sufficient to bound the total branching of $\Psi$. By the results of~\cite{Nadirashvili_mult} the multiplicity of $k$-th eigenvalue is bounded by a linear function of $\gamma$ and $k$. Since the components of $\Psi$ are the $k$-th eigenfunctions of $(M,h)$ one has $n\leqslant C'(\gamma+k)$. This completes the proof if $\Psi$ is not totally isotropic.

If $\Psi$ is totally isotropic, then after rescaling of the metric $h$ one can assume that $\lambda_k(M,h) = 2$. For such metric $g$ one has $E_h(\Psi) = \area_h(M)$. Therefore, by Remark~\ref{Korevaar_remark} 
$$
\Lambda_k(M,h) = 2E_h(\Psi)\leqslant Ck(\gamma + 1).
$$
Combining this inequality with Theorem~\ref{Branching_th} completes the proof.
\end{proof}

%
%
%
%
%
%
%
%

\end{document}